\documentclass[11pt,a4paper]{amsart}
\usepackage{graphicx,multirow,array,amsmath,amssymb,xcolor,caption}

\newtheorem{theorem}{Theorem}

\newtheorem{lemma}[theorem]{Lemma}

\newtheorem*{conjecture}{Conjecture}

\newtheorem*{thm1}{Theorem~\ref{thm:main}}
\newtheorem*{thm2}{Theorem~\ref{thm:super}}
\newtheorem*{thm3}{Theorem~\ref{thm:main2}}

\begin{document}

\title{Petal number of torus knots using superbridge indices}

\author[H. Kim]{Hyoungjun Kim}
\address{College of General Education, Kookmin University, Seoul 02707, Korea}
\email{kimhjun@kookmin.ac.kr}

\author[S. No]{Sungjong No}
\address{Department of Mathematics, Kyonggi University, Suwon 16227, Korea}
\email{sungjongno@kgu.ac.kr}

\author[H. Yoo]{Hyungkee Yoo}
\address{Institute of Mathematical Sciences, Ewha Womans University, Seoul 03760, Korea}
\email{hyungkee@ewha.ac.kr}

\thanks{The first author(Hyoungjun Kim) was supported by the National Research Foundation of Korea (NRF) grant funded by the Korea government Ministry of Science and ICT(NRF-2021R1C1C1012299).
The second author(Sungjong No) was supported by the National Research Foundation of Korea(NRF) grant funded by the Korea government Ministry of Science and ICT(NRF-2020R1G1A1A01101724).
The third author(Hyungkee Yoo) was supported by the National Research Foundation of Korea(NRF) grant funded by the Korea government Ministry of Education(NRF-2019R1A6A1A11051177) and Ministry of Science and ICT(NRF-2022R1A2C1003203).}

\begin{abstract}
A petal projection of a knot $K$ is a projection of a knot which consists of a single multi-crossing and non-nested loops.
Since a petal projection gives a sequence of natural numbers for a given knot,
the petal projection is a useful model to study knot theory.
It is known that every knot has a petal projection.
A petal number $p(K)$ is the minimum number of loops required to represent the knot $K$ as a petal projection.
In this paper, we find the relation between a superbridge index and a petal number of an arbitrary knot.
By using this relation, we find the petal number of $T_{r,s}$ as follows;
$$p(T_{r,s})=2s-1$$
when $1 < r < s$ and $r \equiv 1 \mod s-r$.
Furthermore, we also find the upper bound of the petal number of $T_{r,s}$ as follows;
$$p(T_{r,s})\leq2s- 2\Big\lfloor \frac{s}{r} \Big\rfloor +1$$
when $s \equiv \pm 1 \mod r$.
\end{abstract}

\maketitle

\section{Introduction} \label{sec:intro}
In classical knot theory, knots have been studied with {\it regular projection\/} such that every crossing consists of a single understrand and a single overstrand.
The {\it crossing number\/} of a knot $K$, denoted by $c(K)$, is the least number of such crossings in any projection of $K$.
Adams~\cite{A} introduced $n$-crossings, also known as multi-crossings.
An {\it $n$-crossing\/} is a crossing which consists of $n$ strands such that each strand bisects the crossing.
An {\it $n$-crossing projection\/} is a projection such that all crossings are $n$-crossings.
Figure~\ref{fig:ncro} shows an example of a 5-crossing.
Adams~\cite{A} showed that for every $n \ge 2$, every knot and link has a $n$-crossing projection.
The {\it $n$-crossing number\/} of a knot $K$, denoted by $c_n(K)$, is the least number of $n$-crossings among all $n$-crossing projections of $K$.
Adams et al~\cite{ACDL} showed that for every knot or link $K$, there exists a positive integer $n$ such that $c_n(K)=1$.

\begin{figure}[h!]
\centering
\includegraphics[scale=0.6]{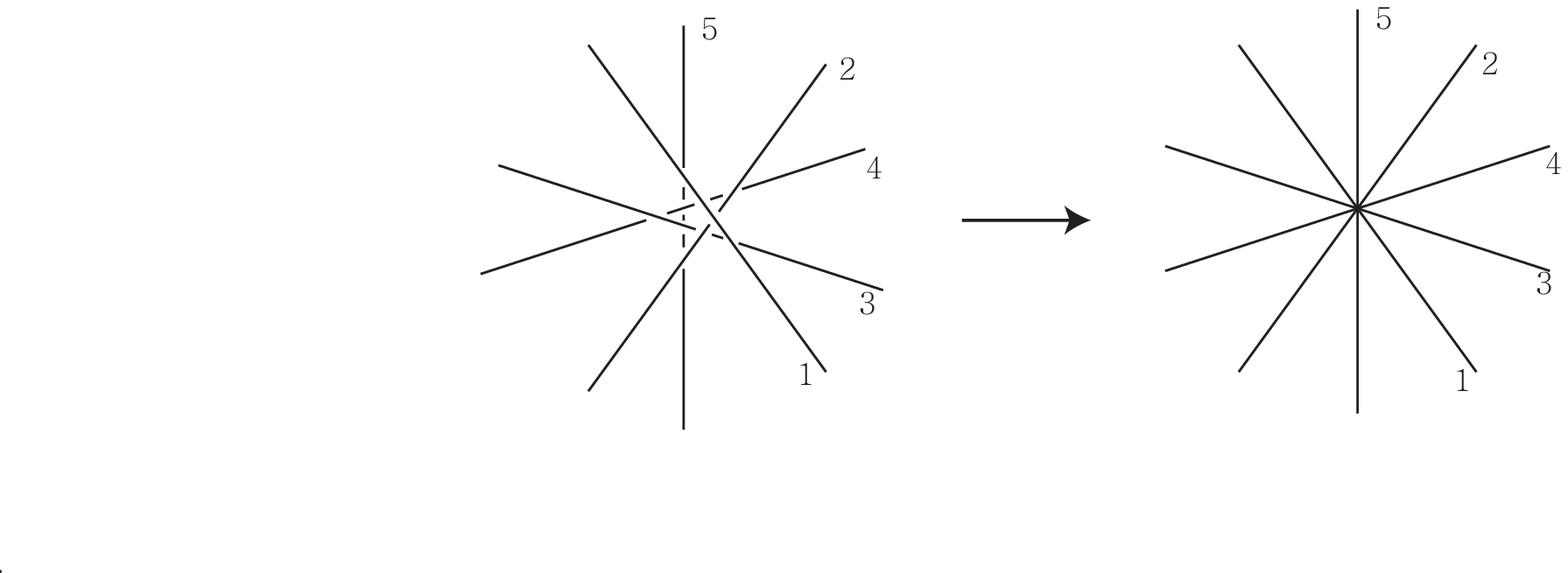}
\caption{An example of a 5-crossing}
\label{fig:ncro}
\end{figure}

An {\it {\"u}bercrossing projection\/} of a knot $K$ is a projection of $K$ with a single $n$-crossing for some $n$.
The {\it {\"u}bercrossing number\/} of a knot $K$, denoted by {\"u}$(K)$, is the smallest $n$ for which there exists a projection with a single $n$-crossing.
An {\"u}bercrossing projection may have {\it nesting loops\/} which contains at least one other loop.
Especially, if an {\"u}bercrossing projection of a knot $K$ has no nesting loop, then it is said to a {\it petal projection\/}.
Figure~\ref{fig:petal} shows an {\"u}bercrossing projection and a petal projection of a trefoil knot.
The {\it petal number\/} of a knot $K$, denoted by $p(K)$, is the smallest number of loops among all petal projections of $K$.
Adams et al~\cite{ACDL} showed that $\text{\"u}(K)$ is strictly less than $p(K)$.
A {\it petal permutation\/} is the permutation of $n$ integers obtained by listing the labels corresponding to the heights of the $n$ strands such that the topmost strand is labelled 1.
Since the petal projection only consists of loops, it cannot represent many links with more than one component.
Therefore, we consider the petal projection and the petal number only for knots.
It is easy to check that the petal number of a nontrivial knot is odd.
Since a petal projection gives a sequence of natural numbers for a given knot,
the petal projection is a useful model to study knot theory.
Even-Zohar et al~\cite{EZHLN} showed that for every nontrivial knot $K$, $p(k) \leq 2c(K) -1$.
Colton et al~\cite{CGHS} proved that
if two petal permutations $\sigma$ and $\sigma'$ represent the same knot type,
then $\sigma$ can be transformed into $\sigma'$
by a sequence of trivial petal additions, trivial petal deletions, and crossing exchanges.

\begin{figure}[h!]
\centering
\includegraphics[scale=1]{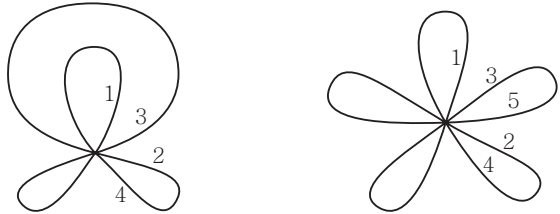}
\caption{The {\"u}bercrossing projection and the petal projection}
\label{fig:petal}
\end{figure}

Let $V$ be a standard solid torus in $S^3$
and $T$ be the boundary of $V$.
The meridian curve on $T$ is a curve that bounds a disc in $V$, 
but does not bound a disc on $T$.
A longitude curve on $T$ is a
curve that intersects a meridian exactly one.
For any relatively prime integers $r$ and $s$,
a {\it torus knot\/} $T_{r,s}$ is a knot on $T$
that wrap around $s$ times in the meridian direction and $r$ times in the longitude direction.

In this paper, we deal with petal numbers of torus knots.
Since $T_{-r,s}$ is a mirror image of $T_{r,s}$,
$p(T_{-r,s})=p(T_{r,s})$.
Thus, without loss of generality, we assume that $r$ and $s$ are positive.
We find the petal number of a particular torus knot as in the following theorem.

\begin{theorem}
\label{thm:main}
Let $r$ and $s$ be relatively prime integers with $1 < r < s$.
If $r \equiv 1 \mod (s-r)$, then
$$p(T_{r,s}) = 2s-1.$$
\end{theorem}

Note that Lee and Jin~\cite{LJ} independently showed that $p(T_{r,r+2}) = 2r+3$ for odd integer $r$ greater than or equal to 3.
To prove Theorem~\ref{thm:main}, we use the relation between a superbridge index and a petal number of a knot.

\begin{theorem}
\label{thm:super}
For any nontrivial knot $K$,
$$
2sb(K) \leq p(K)+1.
$$
\end{theorem}

Adams et al~\cite{ACDL} showed that
$$p(T_{r,s})\leq
\begin{cases}
2s-1  & \text{for} \quad s \equiv 1 \mod r,\\
2s+3  & \text{for} \quad s \equiv -1 \mod r.
\end{cases}$$
We develop this result to the following theorem by using integral surgeries on curves around $T_{r,r+1}$.

\begin{theorem}
\label{thm:main2}
Let $r$ and $s$ be positive integers with $s \equiv \pm 1 \mod r$.
Then
$$p(T_{r,s})\leq
2s- 2\Big\lfloor \frac{s}{r} \Big\rfloor +1.$$
\end{theorem}

The rest of this paper is organized as follows.
In Section~\ref{sec:grid}, we introduce a petal grid diagram and superbridge index.
Also we prove Theorem~\ref{thm:super}.
In Section~\ref{sec:thm1} and \ref{sec:thm3}, we prove Theorem~\ref{thm:main} and \ref{thm:main2}, respectively.
In Section~\ref{sec:conc},
we mention the relation between arc index and petal number for some torus knot.

\section{Petal grid diagrams and superbridge indices}\label{sec:grid}

In this section, we introduce some useful invariants of a knot, and give a relation between them.

\subsection{Petal grid diagrams} \ 

A {\it grid diagram\/} is a knot diagram consisting of a finite number of vertical line segments and the same number of horizontal lines such that every vertical line segment crosses over every horizontal line segment as drawn in Figure~\ref{fig:gridarc}~(a).
Cromwell~\cite{C} showed that every knot has a grid diagram. 
The {\it grid index $g(K)$\/} of a knot $K$ means the minimum number of vertical line segments (or horizontal line segments) among all grid diagrams which represent the knot type $K$.

An {\it arc presentation\/} of a knot $K$ is an embedding of $K$ into $R^3$ which is contained in a finite number of half-planes centered on a binding axis such that each half-plane intersects $K$ at exactly one arc as drawn in Figure~\ref{fig:gridarc}~(c).
Birman and Menasco~\cite{BM} introduced the concept of arc presentation.
The {\it arc index $\alpha(K)$\/} of a knot $K$ means the minimum number of arcs among all arc presentations which represent the knot type $K$.
A grid diagram can be expressed as an arc representation and vice versa as drawn in Figure~\ref{fig:gridarc}~(b). 
In detail, each horizontal line segment in Figure~\ref{fig:gridarc} (a) corresponds to the pair of connected blue line segments in Figure~\ref{fig:gridarc} (b).
Thus the grid index is equal to the arc index for same knot type, i.e. $g(K)= \alpha(K)$.

\begin{figure}[h!]
\centering
\includegraphics[scale=0.9]{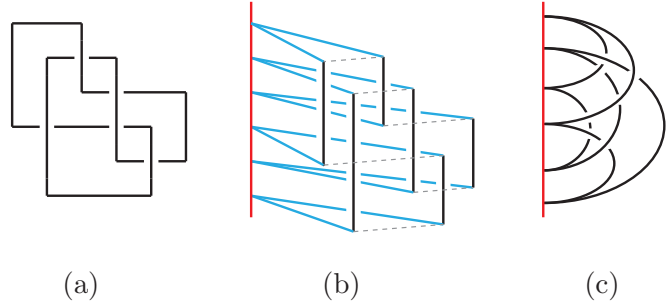}
\caption{A grid diagram and an arc presentation}
\label{fig:gridarc}
\end{figure}

To explain the relationship between the arc index and the petal number, we transform an arc presentation as following.
First consider an arc presentation using a binding circle which is obtained from a binding axis by connecting at the infinity point.
Replace each arc by a straight chord, and label every endpoint of straight chords in clockwise order as drawn in Figure~\ref{fig:petalarc}~(a).
During this process, the order of each arc is not changed.
A {\it petal arc presentation\/} is an arc presentation with $2n+1$ arcs such that every arc connects two endpoints $i$ and $i+n \mod (2n+1)$.
Since every petal projection can be transformed to a petal arc presentation as drawn in Figure~\ref{fig:petalarc}, a petal number is greater than or equal to an arc index for the same knot type, i.e. $p(K) \geq \alpha(K)$.

Similar to a petal arc presentation, a {\it petal grid diagram\/} is a grid diagram with $2n+1$ vertical line segments such that every vertical line segments has length $n$ or $n+1$.
In here, each vertical line segment connects the $i$-th and $(i+n)$-th horizontal line segments with respect to modulo $(2n+1)$.
This means that a petal grid diagram corresponds to a petal arc presentation.
Thus it can be expressed as a petal projection.

\begin{figure}[h!]
\centering
\includegraphics{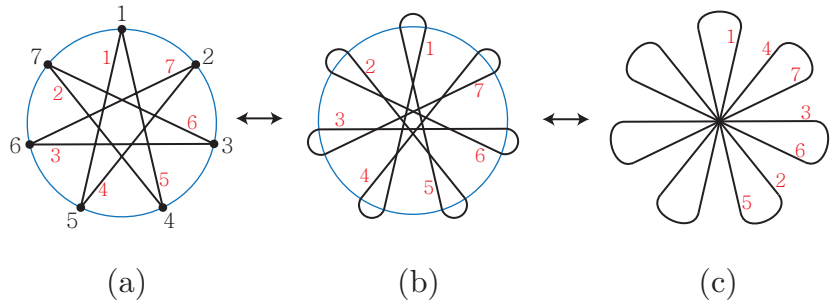}
\caption{A petal projection and corresponding arc presentation}
\label{fig:petalarc}
\end{figure}

\subsection{Superbridge indices} \ 

In~\cite{K}, Kuiper introduces a useful knot invariant, the superbridge index.
We interpret a knot $K$ as an polygonal or smooth embedding $K: S^1 \hookrightarrow \mathbb{R}^3$.
Let $v: \mathbb{R}^3 \rightarrow \mathbb{R}$ be a linear function with norm $\|v\|=1$
and let $\vec{v}$ be the corresponding unit vector.
Let $b_{\vec{v}}(K)$ be the number of local maxima of the restriction to $K$, $v \circ K$.
Then the {\it superbridge index\/} $sb(K)$ is defined by
$$
sb(K)=\min_{K' \in [K]}\max_{\vec{v} \in S^2} b_{\vec{v}}(K')
$$
where $[K]$ is an ambient isotopy class of a knot $K$.
We remark that the {\it bridge index\/} $b(K)$ of $K$ is defined by
$$
b(K)=\min_{K' \in [K]}\min_{\vec{v} \in S^2} b_{\vec{v}}(K').
$$

Adams et al~\cite{ASO} showed that
$2b(K)\leq \text{\"u}(K)$.
For any nontrivial knot $K$,
$\text{\"u}(K) < p(K)$.
So we obtain that $2b(K) \leq p(K)-1$.
Kuiper showed that $b(K)$ is strictly less than $sb(K)$.
Thus we might guess that an upper bound for $sb(K)$ is at least $\frac{p(K)+1}{2}$.

\begin{thm2}
For any nontrivial knot $K$,
$$2sb(K) \leq p(K)+1.$$
\end{thm2}

\begin{proof}
A knot $K$ has a polygonal petal projection on the $xy$-plane as drawn in Figure~\ref{fig:super}~(a).
From this projection, we construct a polygonal knot of $K$.
First, we put a $2p(K)$-prism with $p(K)$ layers in $\mathbb{R}^3$ such that both centers of the top and bottom sides of the prism lie on the $z$-axis.
Label each layer as $1,2,\dots,p(K)$ in the order from top to bottom.
There are $p(K)$ line segments which are passing through the origin in Figure~\ref{fig:super}~(a).
Place these line segments passing through the $z$-axis, say horizontal line segments, on the suitable layers of the prism.
In detail, the assigned number of a horizontal line segment must be the same as the assigned number of the layer on which it is placed.
Since we start to construct a polygonal knot from the polygonal petal projection, each of the remaining line segments is contained in one rectangular side of the prism.
So we obtain the polygonal knot of $K$ as drawn in Figure~\ref{fig:super}~(b).

\begin{figure}[h!]
\centering
\includegraphics{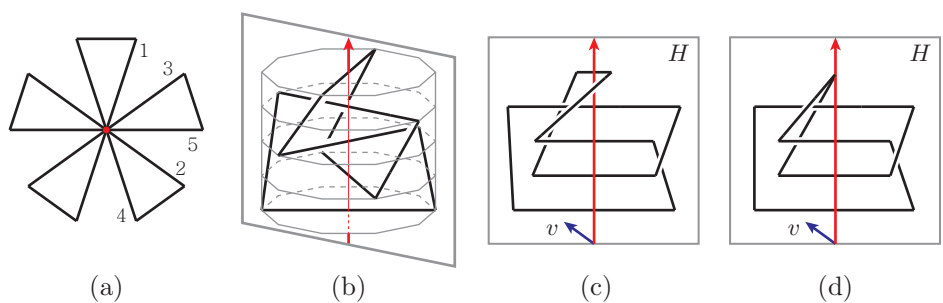}
\caption{The superbridge position of a petal diagram}
\label{fig:super}
\end{figure}

Now we assume that $\vec{v}$ is parallel to the $z$-axis.
Then local maxima and minima can only appear at horizontal line segments.
Note that there are the same number of local maxima and minima.
Since $p(K)$ is odd, there are at most $\frac{p(K)-1}{2}$ local maxima in this case.

It remains to consider the case that $\vec{v}$ is not parallel to the $z$-axis.
Let $H$ be a plane spanned by the two vectors $\vec{v}$ and $\vec{z}=(0,0,1)$.
Then $H$ is divided into two regions $H_1$ and $H_2$ by the $z$-axis.
Now we have a projection of the polygonal knot of $K$ to $H$.

Then there are two cases according to the existence of a horizontal line segment perpendicular to $H$.
The case that there is no horizontal line segment perpendicular to $H$ is described in Figure~\ref{fig:super} (c),
and the other case is described in Figure~\ref{fig:super} (d).
To avoid confusion in the notation, images of line segments which are not horizontal segments are called vertical sticks.
Then there are exactly $p(K)$ vertical sticks.
Note that each vertical stick has at most one local maximum for any direction.
Without loss of generality, we assume that the terminal point of $\vec{v}$ is contained in $H_1$.
Then every vertical stick contained in $H_2$ which does not touch the $z$-axis cannot have a local maximum in $\vec{v}$ direction.
In the case of Figure~\ref{fig:super}~(c), there are $\frac{p(K)-1}{2}$ vertical sticks which are contained completely in $H_2$.
This implies that the number of local maxima is at most $\frac{p(K)+1}{2}$.
In the case of Figure~\ref{fig:super}~(d), there are two cases.
The number of vertical sticks contained in $H_2$ which do not touch the $z$-axis is either $\frac{p(K)-1}{2}$ or $\frac{p(K)-1}{2}-1$.
In the first case, two vertical sticks which are connected by one point on the $z$-axis are contained in $H_1$.
So there are at most $\frac{p(K)+1}{2}$ local maxima.
In the second case, two vertical sticks which are connected by one point in the $z$-axis are contained in $H_2$.
Even though interiors of these vertical sticks are contained in $H_2$, there is at most one local maximum in these sticks.
Thus there are at most $\frac{p(K)+1}{2}$ local maxima in both cases.
Therefore there are at most $\frac{p(K)+1}{2}$ local maxima for any vector $\vec{v}$.
This implies that $2sb(K) \leq p(K)+1$.
\end{proof}

\section{Proof of Theorem~\ref{thm:main}}
\label{sec:thm1}

First we consider a torus knot $T_{r,s}$ with $1<r<s<2r$.
The braid word of $T_{r,s}$ is of the form $(\sigma_1 \sigma_2 ... \sigma_{r-1})^{s}=(\sigma_1 \sigma_2 ... \sigma_{r-1})^{s-r}(\sigma_1 \sigma_2 ... \sigma_{r-1})^{r}$ as drawn in Figure~\ref{fig:braid}~(a).
By the braid relations,
the first $s-r$ strings are fully twisted, and the other $2r-s$ strings are untwisted as drawn in Figure~\ref{fig:braid}~(b).
A full twist of $s-r$ strings can be represented by kinked $s-r-1$ strings wrapping around a string as drawn in Figure~\ref{fig:braid}~(c).

\begin{figure}[h!]
\centering
\includegraphics{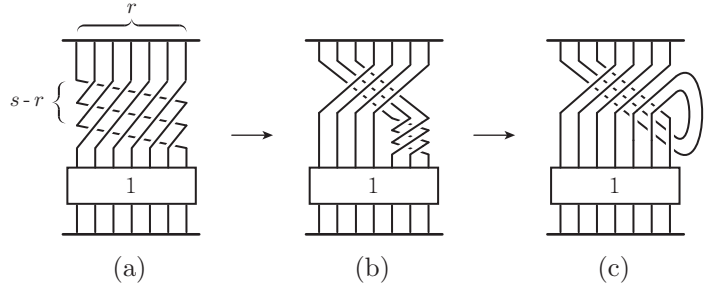}
\caption{The braid of torus knot $T_{r,s}$}
\label{fig:braid}
\end{figure}

We note Kuiper's result~\cite{K}, which says that $sb(T_{r,s})=\min\{2r, s\}$ for $1<r<s$.
Now we recall Theorem~\ref{thm:main}.

\begin{thm1}
Let $r$ and $s$ be relatively prime integers with $1 < r < s$.
If $r \equiv 1 \mod (s-r)$, then
$$p(T_{r,s}) = 2s-1.$$
\end{thm1}

\begin{figure}[h!]
\centering
\includegraphics[scale=0.6]{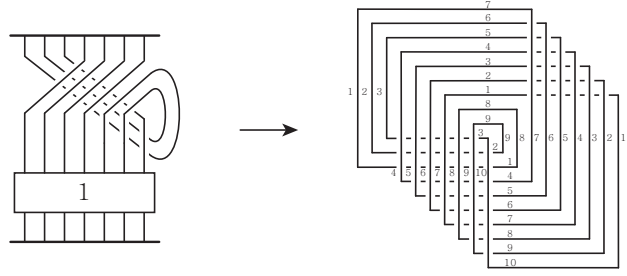}
\caption{The grid diagram for torus knot $T_{r,s}$ with $1<r<s<2r$}
\label{fig:grid}
\end{figure}

\begin{proof}
Let $r$, $s$ be relatively prime integers with $1 < r < s$ and $r \equiv 1 \mod (s-r)$.
This implies that $1<r<s<2r$.
By Theorem~\ref{thm:super} and Kuiper's result, $2s-1 \leq p(T_{r,s})$.
So it is sufficient to show that $p(T_{r,s}) \leq 2s-1$.

Take a grid diagram $D$ of $T_{r,s}$ which has $2s-1$ vertical and horizontal line segments each as drawn in Figure~\ref{fig:grid}.
Label these vertical line segments as $v_1, v_2, \dots ,v_{2s-1}$ in order from left to right.
Similarly, label these horizontal line segments as $h_1, h_2, \dots ,h_{2s-1}$ in order from top to bottom.
Let $D_V$ and $D_H$ be the sets of all vertical and horizontal line segments of $D$, respectively.

Now we define a height function $f$ on $D_V \cup D_H$.
For any vertical line segments $v_i$, we define a height function on $D_V$ as follow:
$$
f(v_i)=
\begin{cases}
i  & \text{if} \quad 1 \leq i < s,\\
2s-i  & \text{if} \quad s \leq i \leq 2s-1
\end{cases}
\quad \text{for} \quad v_i \in D_V.
$$  
Since a horizontal line segment $h_k$ connects two vertical segments $v_i$ and $v_j$ for $1 \leq i < s$ and $s \leq j \leq 2s-1$,
we define a height function on $D_H$ as follow:
$$
f(h_k)=
\begin{cases}
2s-j  & \text{if} \quad 1 \leq k < s,\\
i  & \text{if} \quad s \leq k \leq 2s-1
\end{cases}
\quad \text{for} \quad h_k \in D_H.
$$  
Then the vertical and horizontal line segments that intersect at the north-east corner have the same height except the corner connected to $v_s$ and $h_s$.
Also, the vertical and horizontal line segments that intersect at the south-west corner have the same height.

We remark that $1 < r < s$ and $r \equiv 1 \mod (s-r)$.
Assume that $m=s-r$ and $n= \lfloor \frac{s}{m} \rfloor$.
Choose $2n+1$ pairs of line segments $v_{s \pm lm}$ and $h_{s \pm lm}$ for $l=\{0,\dots,n\}$.
Connect these line segments, then we obtain a path 
$$v_{s-nm} h_{s-nm} v_{s+m} h_{s+m} v_{s-(n-1)m} h_{s-(n-1)m} \cdots v_{s+nm} h_{s+nm} v_{s} h_{s}$$
as a red spiral in Figure~\ref{fig:grid_seq}~(a).
Rearrange $nm$ horizontal line segments $h_i$ for $i < s$ as in the order of
\begin{center}
\begin{tabular}{llll}
    $m$, & $2m$, & \dots, & $nm$,\\
    $m-1$, & $2m-1$, & \dots, & $nm-1$\\
    & & \vdots &\\
    $2$, & $m+2$, & \dots, & $(n-1)m+2$,\\
    $1$, & $m+1$, & \dots, & $(n-1)m+1$.\\
\end{tabular}
\end{center}
Also rearrange $nm$ horizontal line segments $h_i$ for $i > s$ as in the order of 
\begin{center}
\begin{tabular}{lllll}
    & & $s+2m-1$, & \dots, & $s+nm-1$,\\
    & $s+m-2$, & $s+2m-2$, & \dots, & $s+nm-2$,\\
    & & \ \ \ \ \ \ \ \vdots & & \\
    & $s+2$, & $s+m+2$, & \dots, & $s+(n-1)m+2$,\\
    & $s+1$, & $s+m+1$, & \dots, & $s+(n-1)m+1$,\\
    $s+m-1$, & $s+m$, & $s+2m$, & \dots, & $s+nm$.\\
\end{tabular}
\end{center}
During this process, there can exist a horizontal line segment $h_i$ for $i>s$ which cannot move its location changed preserving the knot type when there is a horizontal line segment $h_k$ such that $k>i$ and $f(h_k) \leq f(v_i)$.
In this case, we divide $v_j$ as a two parts $v_j^u$ and $v_j^l$ such that $f(v_j^u)=f(v_j)$ and $f(v_j^l)=f(h_i)$.
Then we obtain a diagram as drawn in Figure~\ref{fig:grid_seq}~(b).

Finally, we adjust vertical line segments of this diagram to construct a petal grid diagram.
Rearrange $nm$ vertical line segments $v_i$ for $i < s$ as in the order of
\begin{center}
\begin{tabular}{lllll}
&    $m$, & $2m$, & \dots, & $(n-1)m$,\\
&    $m-1$, & $2m-1$, & \dots, & $(n-1)m-1$,\\
&    & \ \ \ \ \ \ \ \vdots & & \\
&    $2$, & $m+2$, & \dots, & $(n-2)m+2$, \\
&    $1$, & $m+1$, & \dots, & $(n-2)m+1$, \\
    $s-m$, & $s-m+1$, & $s-m+2$, & \dots, & $s-2$, \ \ \ \  $s-1$.\\
\end{tabular}
\end{center}
Also rearrange $nm$ vertical line segments $v_i$ for $i > s$ as in the order of 
\begin{center}
\begin{tabular}{lllll}
    & & $s+2m-1$, & \dots, & $s+nm-1$,\\
    & $s+m-2$, & $s+2m-2$, & \dots, & $s+nm-2$,\\
    & & \ \ \ \ \ \ \ \vdots & & \\
    & $s+2$, & $s+m+2$, & \dots, & $s+(n-1)m+2$,\\
    & $s+1$, & $s+m+1$, & \dots, & $s+(n-1)m+1$,\\
    $s+m-1$, & $s+m$, & $s+2m$, & \dots, & $s+nm$.\\
\end{tabular}
\end{center}
Thus we obtain a petal grid diagram of $T_{r,s}$ with $2s-1$ vertical line segments as drawn in Figure~\ref{fig:grid_seq}~(c).
Therefore $p(T_{r,s}) \leq 2s-1$.
\end{proof}

\begin{figure}[h!]
\centering
\includegraphics[scale=0.5]{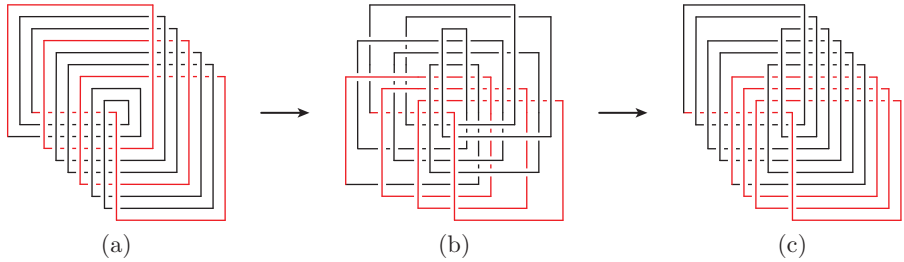}
\caption{The petal grid diagram from the grid diagram for torus knot $T_{r,s}$ with $1 < r < s$ and $r \equiv 1 \mod s-r$}
\label{fig:grid_seq}
\end{figure}

\section{Integral surgeries and proof of Theorem~\ref{thm:main2}} \label{sec:thm3}

Let $V$ be a standard solid torus in $S^3$
and let $T$ be the boundary of $V$.
We assume that $l$ is the core of $V$,
and $m$ is the curve outside $V$ obtained by a small perturbation of the meridian curve of $T$.
Consider the torus knot $T_{r,r+1}$ on $T$ for any positive integer $r$.
If $n$ is an integer with $n \geq 2$,
then $T_{r,nr+1}$ is obtained by a $\frac{1}{n-1}$-surgery on $m$.
Remark that
the surgery does not change $S^3$.
Similarly, $T_{nr-1,r}$ is obtained from $T_{r-1,r}$ after a $\frac{1}{n-1}$-surgery on $l$.
We can put a torus knot $T_{r,r+1}$ in a solid torus $V$ as drawn in Figure~\ref{fig:ml}~(a).
Then curves $l$ and $m$ are represented as drawn in Figure~\ref{fig:ml}~(b).

\begin{figure}[h!]
\centering
\includegraphics{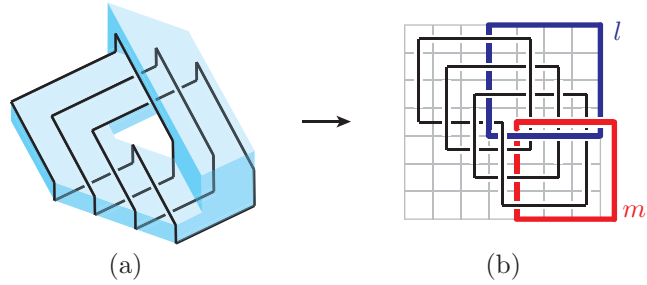}
\caption{The curves $m$ and $l$ for surgeries}
\label{fig:ml}
\end{figure}

Using these facts, we prove the following two lemmas.

\begin{lemma}
\label{lem:nq+1}
For any positive integers $r$ and $n$,
$$p(T_{r,nr+1}) \leq 2(r-1)n+3.$$
\end{lemma}

\begin{proof}
We first remark that $T_{r,nr+1}$ is obtained from $T_{r,r+1}$ by using a $\frac{1}{n-1}$-surgery on $m$.
Take a grid diagram of $T_{r,r+1}$ as drawn in Figure~\ref{fig:lem1}~(a).
In this diagram, there are $r$ strands which intersect right endpoints of the horizontal lines and the bottom endpoints of the vertical lines
(arcs with south-east corners in the figure).
Take a closed curve which wraps around these $r$ strands.
Then this curve corresponds to $m$ for $T_{r,r+1}$.
Now, we use a $\frac{1}{n-1}$-surgery on $m$ at the grid diagram of $T_{r,r+1}$ to obtain $T_{r,nr+1}$.
Then we obtain a grid diagram with $r-1$ spirals as drawn in Figure~\ref{fig:lem1}~(b).
Finally we obtain the petal grid diagram by adjusting the height of horizontal line segments as drawn in Figure~\ref{fig:lem1}~(c).
During this process, an additional $2(n-1)$ vertical segments are needed for each of $r-1$ strands.
Thus we have the result as follows;
$$p(T_{r,nr+1}) \leq 2r + 1 + 2(r-1)(n-1) = 2(r-1)n+3.$$
\end{proof}

\begin{figure}[h!]
\centering
\includegraphics{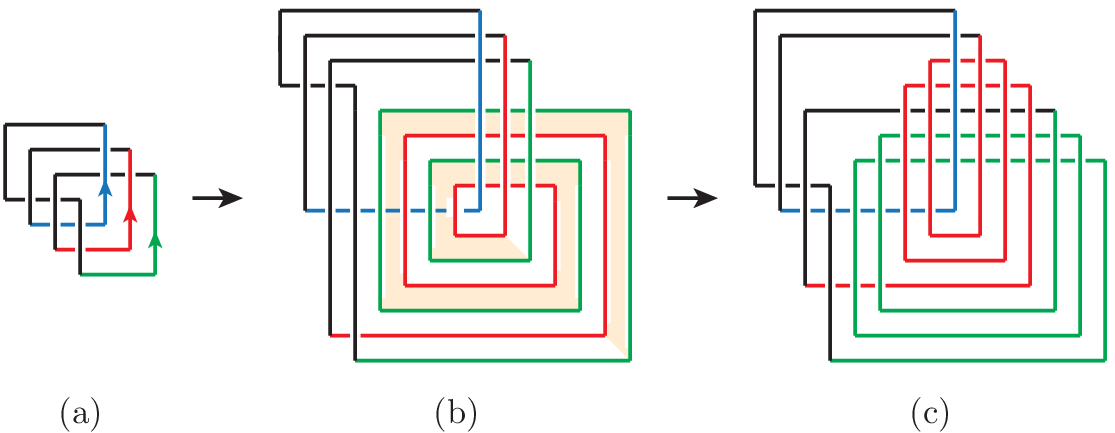}
\caption{The petal grid diagram of $p(T_{r,nr+1})$}
\label{fig:lem1}
\end{figure}

\begin{lemma}
\label{lem:nq-1}
For any positive integers $r$ and $n$,
$$p(T_{r,nr-1}) \leq 2(r-1)n+1.$$
\end{lemma}

\begin{proof}
The process of this proof is similar to the proof of Lemma~\ref{lem:nq+1}.
Since $T_{r,nr-1}$ and $T_{nr-1,r}$ have the same knot type,
we regard that $T_{r,nr-1}$ is obtained from $T_{r-1,r}$ by using a $\frac{1}{n-1}$-surgery on $l$.
Take a grid diagram of $T_{r-1,r}$ as drawn in Figure~\ref{fig:lem2}~(a).
In this diagram, there are $r+1$ strands which intersect right endpoints of the horizontal lines and the top endpoints of the vertical lines
(arcs with north-east corners in the figure).
Take a closed curve which wraps around these $r+1$ strands.
Then this curve corresponds to $l$ for $T_{r-1,r}$.
Now, we use a $\frac{1}{n-1}$-surgery on $l$ at the grid diagram of $T_{r-1,r}$ to obtain $T_{r,nr-1}$.
Then we obtain a grid diagram with $r-1$ spirals as drawn in Figure~\ref{fig:lem2}~(b).
Finally we obtain the petal grid diagram by adjusting the height of horizontal line segments as drawn in Figure~\ref{fig:lem2}~(c).
During this process, an additional $2(n-1)$ vertical segments are needed for each of $r-1$ strands.
Thus we have the result as follows;
$$p(T_{r,nr-1}) \leq 2r - 1 + 2(r-1)(n-1) = 2(r-1)n+1.$$
\end{proof}

\begin{figure}[h!]
\centering
\includegraphics{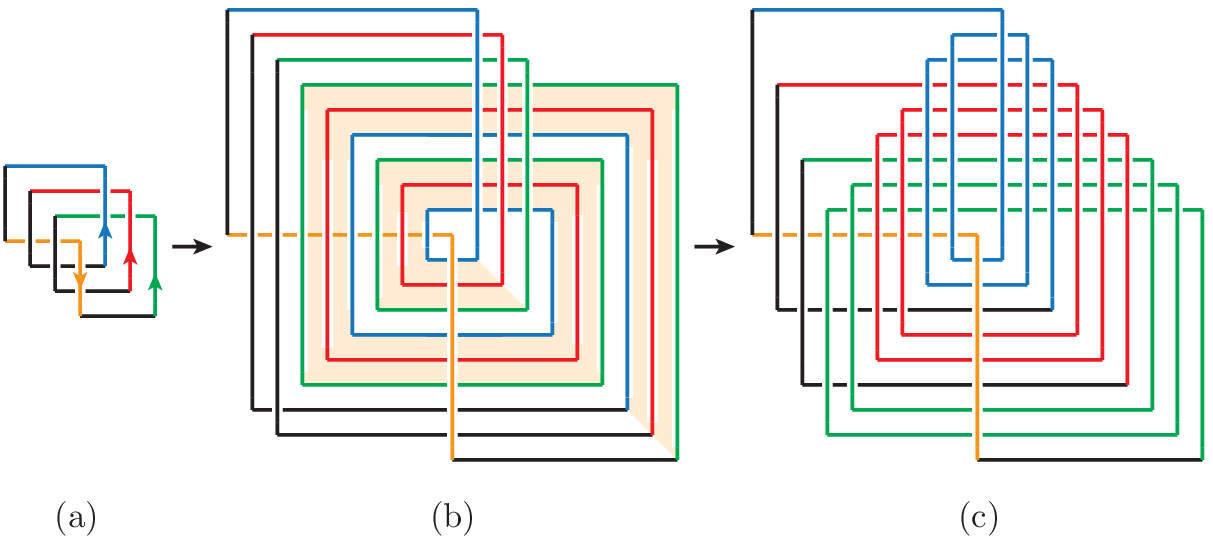}
\caption{The petal grid diagram of $p(T_{r,nr-1})$}
\label{fig:lem2}
\end{figure}

Now we recall Theorem~\ref{thm:main2}.

\begin{thm3}
Let $r$ and $s$ be positive integers with $s \equiv \pm 1 \mod r$.
Then
$$
p(T_{r,s})\leq
2s- 2\Big\lfloor \frac{s}{r} \Big\rfloor +1.
$$
\end{thm3}

\begin{proof}
First assume that $s \equiv 1 \mod r$.
Then $s=\lfloor \frac{s}{r} \rfloor r + 1$.
By Lemma~\ref{lem:nq+1}, 
\begin{align*}
    p(T_{r,s}) \ = \ p(T_{r,\lfloor \frac{s}{r} \rfloor r + 1}) \ & \leq \ 2(r-1) \Big\lfloor \frac{s}{r} \Big\rfloor +3 \\
    & = \ 2 \left( r \Big\lfloor \frac{s}{r} \Big\rfloor +1 \right) - 2 \Big\lfloor \frac{s}{r} \Big\rfloor + 1\\
    & = \ 2s - 2 \Big\lfloor \frac{s}{r} \Big\rfloor +1.\\
\end{align*}
It remains to consider the case when $s \equiv -1 \mod r$.
Then $s=(\lfloor \frac{s}{r} \rfloor+1) r - 1$.
By Lemma~\ref{lem:nq-1}, 
\begin{align*}
    p(T_{r,s}) \ = \ p(T_{r,(\lfloor \frac{s}{r} \rfloor+1) r - 1}) \ & \leq \ 2 \left( r-1 \right) \left( \Big\lfloor \frac{s}{r} \Big\rfloor +1 \right) +1 \\
    & = \ 2 r \left( \Big\lfloor \frac{s}{r} \Big\rfloor +1 \right) - 2 - 2 \Big\lfloor \frac{s}{r} \Big\rfloor + 1\\
    & = \ 2s - 2 \Big\lfloor \frac{s}{r} \Big\rfloor +1.\\
\end{align*}
\end{proof}

\section{Concluding remark}
\label{sec:conc}

In this section, we discuss the relation between $\alpha(K)$ and $p(K)$.
Let $r$ and $s$ be relatively prime integers with $1 < r < s$.
Etnyre and Honda~\cite{EH} showed that an arc index of a torus knot $T_{r,s}$ is equal to $r+s$.

First consider the case that $r \equiv 1 \mod (s-r)$.
Then for a torus knot $T_{r,s}$, $\alpha(T_{r,s}) = r+s$ and $p(T_{r,s})=2s-1$ by Theorem~\ref{thm:main}.
Therefore, there is a difference of as much as $s-r-1$ between the petal number and the arc index for a torus knot $T_{r,s}$.
This means that there is no gap between the arc index and the petal number for a torus knot $T_{r,s}$ when $r$ is equal to $s-1$.
Even though there is a gap of 1 between the arc index and the petal number for a torus knot $T_{r,s}$ when $r$ is equal to $s-2$, it is meaningless.
Because the gap comes from the fact that a petal number is odd for every knot.

Now consider the case that $r=2$.
Then for a torus knot $T_{2,s}$,
$$p(T_{2,s})\leq 2s- 2\Big\lfloor \frac{s}{2} \Big\rfloor +1 = 2s-(s-1)+1 = s+2$$ 
by Theorem~\ref{thm:main2}.
Since an arc index of $T_{2,s}$ is equal to $2+s$, the petal number of $T_{2,s}$ is also equal to $2+s$.
We remark that $T_{2,s}$ is a 2-bridge knot.
Adams et al~\cite{ACDL} showed that a petal number of a 2-bridge knot $K$ with an odd number of crossings is equal to $c(K)+2$.
From this result, we also have the same result of petal number of $T_{2,s}$ as obtained from Theorem~\ref{thm:main2}, since a torus knot $T_{2,s}$ is 2-bridge knot with $s$ crossings.
Furthermore if $r$ and $s$ satisfy the condition of Theorem~\ref{thm:main}, then $\lfloor \frac{s}{r} \rfloor =1$. 
This means that
$$p(T_{r,s}) = 2s-1=2s-2\Big\lfloor \frac{s}{r} \Big\rfloor+1.$$
So we guess that Theorem~\ref{thm:main2} holds for every torus knot $T_{r,s}$.
Thus we suggest the following conjecture.

\begin{conjecture}
Let $r$ and $s$ be relatively prime integers with $1 < r < s$.
Then
$$
p(T_{r,s})\leq
2s- 2\Big\lfloor \frac{s}{r} \Big\rfloor +1.
$$
\end{conjecture}

\bibliography{petaltorus.bib} 
\bibliographystyle{siam}

\end{document}